\documentclass[11pt]{article}
\usepackage{graphicx,indentfirst}
\usepackage{amscd}
\usepackage[all]{xy}
\usepackage{amsmath}
\usepackage{enumerate}
\usepackage{amsfonts,amssymb}
\usepackage{extarrows}
\usepackage{amsthm}
\usepackage{latexsym}
\usepackage[american]{babel}
\usepackage{times}
\usepackage{tikz}
\usepackage{tikz-cd}
\usetikzlibrary{arrows, decorations.markings}
\usepackage{url}
\usepackage{appendix}
\usepackage{bookmark}
\usepackage{verbatim}
\usepackage{mathrsfs} 
\usepackage{MnSymbol}
\usepackage[affil-it]{authblk}

\usetikzlibrary{matrix,arrows,decorations.pathmorphing}
\usepackage[a4paper,left=2.5cm,right=2.5cm,bottom=3cm,top=3.5cm]{geometry}

\newtheorem{theorem}{Theorem}[section]
\newtheorem{lemma}[theorem]{Lemma}
\newtheorem{prop}[theorem]{Proposition}
\newtheorem{cor}[theorem]{Corollary}

\theoremstyle{definition}\newtheorem{defn}{Definition}[section]
\theoremstyle{remark}\newtheorem{example}{Example}
\theoremstyle{remark}\newtheorem{rem}{Remark}
\theoremstyle{remark}
\theoremstyle{remark}

\begin{document}

\newcommand{\ad}{\text{ad}}
\newcommand{\tr}{\text{tr}}
\newcommand{\bDelta}{\trueroots}
\newcommand{\Hom}{\text{Hom}}
\newcommand{\op}{\text{op}}
\newcommand{\id}{\text{id}}
\newcommand{\pr}{\text{pr}}
\newcommand{\simto}{\overset{\sim}{\to}}
\newcommand{\For}{\text{For}}
\newcommand{\Chi}{\mathcal{X}}
\newcommand{\Res}{\text{Res}}
\newcommand{\roots}{\mathbf{Q}}
\newcommand{\trueroots}{\mathbf{\Delta}}
\newcommand{\positiveroots}{\mathbf{\Delta}^+}
\newcommand{\simpleroots}{\mathbf{\Sigma}}
\newcommand{\weights}{\mathbf{P}}
\newcommand{\wt}{\mathsf{wt}} 
\newcommand{\ZZ}{\mathbb{Z}}
\renewcommand*{\O}{\mathcal{O}}
\newcommand{\ch}{\text{ch}} 

%\runningtitle{Tensor-closed objects in category $\O$}
\title{Tensor-closed objects in the BGG category of a quantized semisimple Lie algebra}
%\shorttitle{Twisted complexes as dg-enhancements}
\author{Zhaoting Wei\footnote{Kent State University at Geauga, 14111 Claridon Troy Road, Burton, OH 44021, USA, \texttt{zwei3@kent.edu}}}

%\authorheadline{Z. Wei}
\maketitle
\begin{abstract}
We consider the BGG category $\O$ of a quantized universal enveloping algebra  $U_q(\mathfrak{g})$. We call a module $M\in \O$  tensor-closed if $M\otimes N\in\O$ for any $N\in \O$. In this paper we prove that $M\in \O$ is tensor-closed if and only if $M$ is finite dimensional. The method used in this paper applies to the unquantized case as well.
\end{abstract}



\section*{Acknowledgment}
James Humphreys told the author to use formal characters to show certain modules are not in $\O$. Peter McNamara told the author to consider the rational form of formal characters. Victor Ostrik provided the counter-example in Remark \ref{rmk: counter example}. The author wants to thank them all. The author also want to thank Weibo Fu for kindly answering questions related to this paper.

\section{Introduction}
BGG category $\O$ plays a central role in representation theory, see \cite{humphreys2008representations}. For a complex semisimple Lie algebra $\mathfrak{g}$ we can consider its quantized universal enveloping algebra $U_q(\mathfrak{g})$ and the category $\O$ of $U_q(\mathfrak{g})$ as in \cite{andersen2011category} and  \cite{voigt2017complex}.

The large category $U_q(\mathfrak{g})$-Mod has a tensor product but category $\O$ is not closed under the tensor product.  We call a module $M\in \O$  \emph{tensor-closed} if $M\otimes N\in\O$ for any $N\in \O$. It is easy to show that finite dimensional modules are tensor-closed. Actually in \cite{lunts1999localization} used tensor products of finite dimensional  $U_q(\mathfrak{g})$-modules to construct the coordinate ring of the deformed flag variety of $\mathfrak{g}$. Therefore it is interesting to ask whether we can characterize finite dimensional $U_q(\mathfrak{g})$-modules in a categorical way, to which we give an affirmative answer in this paper.

For (unquantized) complex semisimple Lie algebra $\mathfrak{g}$ it is  a folklore  theorem that any tensor-closed module in $\O$ must be finite dimensional, see \cite{humphreys2015tensorclosed} for an outline of the proof.

The main result of this paper is Theorem \ref{thm: tensor-closed is finite dimensional}, which claims that  $M\in \O$ of  $U_q(\mathfrak{g})$ is tensor-closed if and only if $M$ is finite dimensional. This result gives a categorical characterization of finite dimensional modules in category $\O$. The proof is based on the idea in \cite{humphreys2015tensorclosed} together with a careful study of rational expressions of formal characters of modules in $\O$. We can apply the same proof to the unquantized case with little modification.

\section{A review of the BGG category $\O$ of a quantized universal enveloping algebra}
\subsection{A review of quantized universal enveloping algebras}
We follow the notations in \cite{voigt2017complex}. Please also see \cite{andersen2011category} for  references. Let $\mathfrak{g}$ a semisimple Lie algebra over $ \mathbb{C} $ of rank $ N $. 
We fix a Cartan subalgebra $ \mathfrak{h} \subset \mathfrak{g} $. Let $\trueroots$ be the set of roots  and we fix $ \simpleroots = \{\alpha_1, \dots, \alpha_N\}\subset \trueroots $ the set of simple roots. 
We write $ (\;,\;) $ for the bilinear form on $ \mathfrak{h}^* $ obtained by rescaling the Killing form such that the 
shortest root $ \alpha $ of $ \mathfrak{g} $ satisfies $ (\alpha,\alpha) = 2 $. For a root $\beta\in \trueroots$ 
we set $d_\beta=(\beta,\beta)/2$ and let $\beta^{\vee}=\beta/d_{\beta}$ be the corresponding coroot.
In particular let $ 
d_i = (\alpha_i, \alpha_i)/2
$
and hence 
$
\alpha_i^\vee = d_i^{-1} \alpha_i
$
for  $ i = 1, \dots, N $. 

Denote by $ \varpi_1, \dots, \varpi_N $ the fundamental weights of $ \mathfrak{g} $, 
satisfying the relations $ (\varpi_i, \alpha_j^\vee) = \delta_{ij} $. 
We write 
\begin{equation}
\weights = \bigoplus_{j = 1}^N \mathbb{Z} \varpi_j,
 \qquad \roots = \bigoplus_{j = 1}^N \mathbb{Z} \alpha_j,
 \qquad \roots^\vee = \bigoplus_{j = 1}^N \mathbb{Z} \alpha_j^\vee,
\end{equation}
for the weight, root and coroot lattices of $ \mathfrak{g} $, respectively.  It is well-known that $\beta^{\vee}\in \roots^\vee$ for each $\beta\in \trueroots$. 
 
Let  $ \weights^+ $ denote the set of dominant integral weights  and $ \roots^+ $ denote the set of  non-negative integer combinations of the simple roots. Let $\positiveroots= \roots^+ \cap \trueroots$ be the set of positive roots.

As in the standard notation, let   $(a_{ij})_{1\leq i,j\leq N}$ be the Cartan matrix for $ \mathfrak{g} $  and let $W$ be the Weyl group  for $ \mathfrak{g}$. See \cite[Chapter III]{james1978introduction} for details.

%Throughout we fix the smallest positive integer $ L $ such that the numbers $ (\varpi_i, \varpi_j) $ take values 
%in $ \frac{1}{L} \mathbb{Z} $ for all $ 1 \leq i,j \leq N $.

\begin{defn} \label{def:uqg}[\cite[Definition 2.13]{voigt2017complex}]
Let   $ q = e^h \in \mathbb{R}^\times $ be an invertible element for $h\in \mathbb{R}^{\times}$. It is clear $q$ is not a root of $1$.
The algebra $ U_q(\mathfrak{g}) $ over $ \mathbb{C} $ has generators $ K_\lambda $ 
for $ \lambda \in \weights $, and $ E_i, F_i $ for $ i = 1, \dots, N $, and the defining relations for $ U_q(\mathfrak{g}) $ are 
\begin{equation}
\begin{split}
K_0 = 1,~K_\lambda K_\mu = K_{\lambda + \mu}, &~K_\lambda E_j K_\lambda^{-1} = q^{(\lambda, \alpha_j)} E_j,~K_\lambda F_j K_\lambda^{-1} = q^{-(\lambda, \alpha_j)} F_j, \\ 
&[E_i, F_j] = \delta_{ij} \frac{K_i - K_i^{-1}}{q_i - q_i^{-1}}, 
\end{split}
\end{equation} 
for all $ \lambda, \mu \in \weights $ and all $ i, j $, together with the quantum Serre relations 
\begin{equation}
\begin{split}
&\sum_{k = 0}^{1 - a_{ij}} (-1)^k \begin{bmatrix} 1 - a_{ij} \\ k \end{bmatrix}_{q_i}
E_i^{1 - a_{ij} - k} E_j E_i^k = 0 \\ 
&\sum_{k = 0}^{1 - a_{ij}} (-1)^k \begin{bmatrix} 1 - a_{ij} \\ k \end{bmatrix}_{q_i}
F_i^{1 - a_{ij} - k} F_j F_i^k = 0. 
\end{split}
\end{equation}
In the above formulas we abbreviate $ K_i = K_{\alpha_i} $ for all simple roots, and we use the notation $ q_i = q^{d_i} $. 
\end{defn}

$ U_q(\mathfrak{g}) $ is a Hopf algebra 
with comultiplication $ \hat{\Delta}: U_q(\mathfrak{g}) \rightarrow U_q(\mathfrak{g}) \otimes U_q(\mathfrak{g}) $ given by 
\begin{equation}
\begin{split}
\hat{\Delta}(K_\lambda) &= K_\lambda \otimes K_\lambda, \\
\hat{\Delta}(E_i) &= E_i \otimes K_i + 1 \otimes E_i \\
\hat{\Delta}(F_i) &= F_i \otimes 1 + K_i^{-1} \otimes F_i,  
\end{split}
\end{equation}
counit $ \hat{\epsilon}: U_q(\mathfrak{g}) \rightarrow \mathbb{C} $ given by $
\hat{\epsilon}(K_\lambda) = 1, \hat{\epsilon}(E_j) = 0,  \hat{\epsilon}(F_j) = 0$, 
and antipode $ \hat{S}: U_q(\mathfrak{g}) \rightarrow U_q(\mathfrak{g}) $ given by $
\hat{S}(K_\lambda) = K_{-\lambda},  \hat{S}(E_j) = -E_j K_j^{-1},  \hat{S}(F_j) = -K_j F_j
$.
%We   use the Sweedler notation 
%$ 
%\hat{\Delta}(X) = X_{(1)} \otimes X_{(2)} 
%$ 
%for the coproduct of $ X \in U_q(\mathfrak{g}) $. 

Let $ U_q(\mathfrak{n}_+) $ be the subalgebra of $ U_q(\mathfrak{g}) $ generated by the elements 
$ E_1, \dots, E_N $, and $ U_q(\mathfrak{b}_+) $ be the subalgebra of $ U_q(\mathfrak{g}) $ generated by 
$ E_1, \dots, E_N $ and all $ K_\lambda $ for $ \lambda \in \weights $. We  define $U_q(\mathfrak{n}_-)$ and $ U_q(\mathfrak{b}_-) $ in the same way.
Moreover we let $ U_q(\mathfrak{h}) $ be the subalgebra generated by the elements $ K_\lambda $ 
for $ \lambda \in \weights $.  These algebras are  Hopf subalgebras of $ U_q(\mathfrak{g}) $. By \cite[Proposition 2.14]{voigt2017complex}
we know that there is a linear isomorphism 
\begin{equation}
U_q(\mathfrak{n}_-) \otimes U_q(\mathfrak{h}) \otimes U_q(\mathfrak{n}_+) \cong U_q(\mathfrak{g}).
\end{equation}

\subsection{A review of the BGG category $\O$}
Recall  that $ 1 \neq q = e^h $ for an $h\in \mathbb{R}^{\times}$. 
We shall also use the notation $ \hbar = \tfrac{h}{2 \pi} $ hence $q=e^{2\pi\hbar}$. 

As in \cite[Section 2.3.1]{voigt2017complex} we let $\mathfrak{h}^* = \Hom_{\mathbb{C}}(\mathfrak{h},\mathbb{C})$ and $\mathfrak{h}_q^* = \mathfrak{h}^*/i\hbar^{-1}\roots^\vee$ be the parameter space for weights. Here $i=\sqrt{-1}$. It is clear that there is an embedding $\text{Span}_{\mathbb{R}}\trueroots \subset \mathfrak{h}_q^*$. In particular $\roots\subset\weights\subset \mathfrak{h}_q^*$.

One says that a vector $v$ in    a left $ U_q(\mathfrak{g}) $-module is a weight vector of weight $\lambda\in\mathfrak{h}_q^*$ if it is a common eigenvector for the action of $U_q(\mathfrak{h})$ with
$$
K_\mu \cdot v = q^{(\lambda,\mu)} v, \qquad\qquad \text{for all }\mu\in\weights.
$$
It is well defined: if $\lambda\in i\hbar^{-1}\roots^\vee$ then for any $\mu\in\weights$ we have $q^{(\lambda,\mu)}=e^{2\pi \hbar (\lambda,\mu)}=1$.

\begin{defn} \label{def: categoryO}[\cite[Definition 3.1]{andersen2011category}, \cite[Definition 4.1]{voigt2017complex}]
A left module $ M $ over $ U_q(\mathfrak{g}) $ is said to belong to the \emph{BGG category} $ \O $ if 
\begin{enumerate}[a)]
\item $ M $ is finitely generated as a $ U_q(\mathfrak{g}) $-module. 
\item $ M $ is a weight module, that is, a direct sum of its weight spaces $ M_\lambda $ for $ \lambda \in \mathfrak{h}^*_q $. 
\item The action of $ U_q(\mathfrak{n}_+) $ on $ M $ is locally nilpotent, that is, for each $v\in M$, the subspace $ U_q(\mathfrak{n}_+) \cdot v$ of $M$ is finite dimensional.
\end{enumerate}
Morphisms in category $ \O $ are all $ U_q(\mathfrak{g}) $-linear maps. 
\end{defn} 

We list some basic properties of category $\O$.

\begin{prop}\label{prop: properties of O}
\begin{enumerate}
\item  $\O$ is closed under  submodules, quotient modules, and finite direct sums.
\item All  weight spaces of $M$ in $\O$ are finite dimensional.
\item All finite dimensional weight modules of $ U_q(\mathfrak{g}) $ are in $\O$.
%\item $ \O $ is Artinian and Notherian.
\end{enumerate}
\end{prop}

\begin{defn}\label{def: Verma module}[\cite[Definition 2.31]{voigt2017complex}]
The \emph{Verma module} $ M(\lambda) $ associated to $ \lambda \in \mathfrak{h}^*_q $ is the 
induced $ U_q(\mathfrak{g}) $-module 
\begin{equation}
M(\lambda) = U_q(\mathfrak{g}) \otimes_{U_q(\mathfrak{b}_+)} \mathbb{C}_\lambda
\end{equation}
where $ \mathbb{C}_\lambda $ denotes the one-dimensional $ U_q(\mathfrak{b}_+) $-module $ \mathbb{C} $ with the action induced from the character $ \chi_\lambda $ determined by
\begin{equation}
 \chi_\lambda(K_\mu) = q^{(\lambda,\mu)}  \text{ for all }\mu\in\weights, \text{ and }  \chi_\lambda(E_i) =0, ~i=1,\ldots, N.
\end{equation}
It is clear that $ M(\lambda) $ belongs to category $\O$.

$ M(\lambda) $ contains a unique maximal proper 
submodule $ I(\lambda) $, namely the linear span of all submodules not containing the highest weight  $ v_\lambda = 1 \otimes 1 \in M(\lambda) $. The resulting simple quotient module $ M(\lambda)/I(\lambda) $ will be denoted $ V(\lambda) $. It is again a module in $\O$.
\end{defn}

\begin{rem}In \cite{andersen2011category}, $ M(\lambda)$ and $ V(\lambda) $ are denoted by $\Delta_q(\lambda)$ and $L_q(\lambda)$ respectively.
\end{rem}

It is clear that every highest weight module of highest weight $ \lambda $ is isomorphic to a quotient of $ M(\lambda) $ and every simple highest weight module of highest weight $ \lambda $ is isomorphic to   $ V(\lambda) $ .

%\begin{prop}
% \label{prop:Uqg_characters}[\cite[Proposition 2.34]{voigt2017complex}]
%For every $\omega\in\mathbf{X}_q$, there is a one-dimensional representation
% $\chi_\omega : U_q(\mathfrak{g}) \to \mathbb{K}$ defined on generators by
% \[
  %\chi_\omega(K_\mu) = q^{(\omega, \mu)}, \quad
 % \chi_\omega(E_i) = \chi_\omega(F_i) = 0,
 %\]
% and one-dimensional representation of $U_q(\mathfrak{g})$ is of this form.
%\end{prop}

The following result characterizes finite dimensional weight modules of $ U_q(\mathfrak{g}) $
\begin{prop}\label{prop: finite dimensional modules}[\cite[Corollary 2.100]{voigt2017complex}]
We write $\mathbf{X}_q$ for the set of weights $\omega\in\mathfrak{h}_q^*$ satisfying
 $
  q^{(\omega,\alpha)} = \pm1$  for all $\alpha\in\roots$.
We define 
 \[
  \weights_q^+ = \weights^+ + \mathbf{X}_q \subset \mathfrak{h}_q^*.
 \]
Then every finite dimensional weight module over $U_q(\mathfrak{g})$ decomposes into a direct sum of irreducible highest weight modules $V(\lambda)$ for weights $\lambda\in\weights_q^+$.
\end{prop}

Simple modules $ V(\lambda) $ are the building blocks of modules in $\O$.

\begin{prop}\label{prop: Jordan-Holder}[\cite[Theorem 4.3]{voigt2017complex}]
Every module $ M \in \O $ is both Artinian and Noetherian. Hence every y module $ M \in \O $
has a Jordan-H\"older decomposition series $ 0 = M_0 \subset M_1 \subset \cdots \subset M_n = M $ such that all subquotients $ M_{j + 1}/M_j $ are simple 
highest weight modules. Moreover, the number of subquotients isomorphic to $ V(\lambda) $ for $ \lambda \in \mathfrak{h}^*_q $ is independent 
of the decomposition series
and will be denoted by $ [M: V(\lambda)] $.
\end{prop}

To further study $ [M(\mu): V(\lambda)]  $ for a Verma module $M(\mu)$  we need the following  concept.

\begin{defn} 
	\label{def:extended_Weyl_group}[\cite[Section 8.3.2]{joseph1995quantum}, \cite[Definition 2.125]{voigt2017complex}]
	We define 
	\begin{equation}
	\mathbf{Y}_q = \{\zeta\in\mathfrak{h}_q^* \mid 2\zeta=0\}\cong\frac{1}{2} i\hbar^{-1} \roots^\vee / i\hbar^{-1} \roots^\vee.
	\end{equation}
	It is clear that $W$ acts on $\mathbf{Y}_q$. The \emph{extended Weyl group} $ \hat{W} $ is defined as the semidirect product 
		\begin{equation}
	\hat{W} = \mathbf{Y}_q \rtimes W
	\end{equation} 
	with respect to the action of $ W $ on $ \mathbf{Y}_q $. $\hat{W}$ is a finite group.
\end{defn} 

 Explicitly, the product in $ \hat{W} $ is $(i\zeta, v)(i\eta, w) = (i\zeta + i v \eta, vw)$.
We define two actions of $\hat{W}$ on $\mathfrak{h}_q^*$ by $
  (i\zeta,w)\lambda = w\lambda+i\zeta$
and
\begin{equation}
(i\zeta,w)\cdot\lambda  = w\cdot\lambda+\zeta = w(\lambda+\rho) - \rho +i\zeta,
\end{equation}
for $\lambda\in\mathfrak{h}_q^*$, where $\rho$ is the half sum of all positive roots.  The latter is called the \emph{shifted action} of $\hat{W}$ on $\mathfrak{h}_q^*$.

\begin{rem}
See \cite[Theorem 2.128]{voigt2017complex} for the relation between the $\frac{1}{2} i\hbar^{-1} \roots^\vee$-translation and the \emph{Harish-Chandra map}, which  plays an important role in the representation theory of $U_q(\mathfrak{g})$.
\end{rem}

\begin{defn} \label{defwlinkage}
We say that $ \mu, \lambda \in \mathfrak{h}^*_q $ are $ \hat{W} $-linked if $ \hat{w} \cdot \lambda = \mu $ 
for some $ \hat{w} \in \hat{W} $. 
\end{defn}

\begin{defn}\label{def: partial order on weights}
	We define a partial order $ \geq $ on $ \mathfrak{h}^*_q $ by saying that $ 
	\lambda \geq \mu$ if $ \lambda - \mu \in \roots^+ $. 
	Here we are identifying $\roots^+$ with its image in $ \mathfrak{h}^*_q $. 
\end{defn}

\begin{lemma}\label{lemma: multiplicities of Jordan-Holder}[\cite[Section 4.1.1]{voigt2017complex}]
For any $ \mu \in \mathfrak{h}^*_q $ we  have $ [M(\mu): V(\mu)] = 1 $, and moreover $ [M(\mu): V(\lambda)] = 0 $ unless $\lambda\leq \mu$ and $\lambda$ is $\hat{W}$-linked to $\mu$. Since $\hat{W}$ is a finite group, for each $ \mu \in \mathfrak{h}^*_q $ there exists only finitely many $\lambda\in \mathfrak{h}^*_q$ such that $ [M(\mu): V(\lambda)] \neq 0 $.
\end{lemma}

\section{Formal characters of modules in category $\O$}
\subsection{Basic properties of formal characters}

\begin{defn}\label{def: formal characters}
We define the formal character of $ M $ in $\O$ as the formal sum
\begin{equation}
\ch(M) = \sum_{\lambda \in \mathfrak{h}^*_q} \dim(M_\lambda) e^\lambda.
\end{equation}
\end{defn}

By Proposition \ref{prop: properties of O} any module $ M $ in category $ \O $ satisfies $ \dim M_\lambda < \infty $ for all $ \lambda \in \mathfrak{h}^*_q $. So $\ch(M)$  is well-defined. We also have the following more general definition:

\begin{defn}\label{def: convolution ring}
Let $\Chi$  be the ring of formal sums of the form $\sum_{\lambda\in\mathfrak{h}_q^*} f(\lambda) e^\lambda$ where $f:\mathfrak{h}^*_q \to \ZZ$ is any integer valued function whose support lies in a finite union of sets of the form $\nu - \roots^+$ with $\nu\in\mathfrak{h}_q^*$.  The product in $\Chi$ is the  convolution product given by
\[
  \left(\sum_{\lambda\in\mathfrak{h}_q^*} f(\lambda)e^\lambda \right)
  \left(\sum_{\mu\in\mathfrak{h}_q^*} g(\mu)e^\mu \right)   = \sum_{\lambda,\mu\in\mathfrak{h}_q^*} f(\lambda) g(\mu) e^{\lambda+\mu}.
\]
It is clear that the right hand side is still in  $\Chi$.
\end{defn}

\begin{defn}\label{Kostant partition function}
We introduce an element $p\in \Chi$ as 
\begin{equation}
   p = \prod_{\beta\in \positiveroots} \left(\sum_{m=0}^\infty e^{-m\beta}\right) .
\end{equation}
\end{defn}

\begin{lemma}\label{lemma: formal character of Verma modules}[\cite[Proposition 2.68]{voigt2017complex}]
For each $ \mu \in \mathfrak{h}^*_q $, the formal character of the Verma module  $ M(\mu) $ is the convolution product of $e^\mu$ and $p$:
\begin{equation}
  \ch(M(\mu)) = e^\mu p.
\end{equation}
\end{lemma}

By Lemma \ref{lemma: multiplicities of Jordan-Holder} we have
\begin{equation}
\ch(M(\mu)) = \sum_{\lambda \in \mathfrak{h}^*_q} [M(\mu): V(\lambda)] \ch(V(\lambda)). 
\end{equation}
where  $ [M(\mu): V(\mu)] = 1 $ and $ [M(\mu): V(\lambda)] = 0 $ unless $\lambda\leq \mu$ and $\lambda$ is $\hat{W}$-linked to $\mu$. 

We can obtain the following well-known result on $\ch(V(\mu))$ by inverting the matrix $ [M(\mu): V(\lambda)] $:
\begin{lemma}\label{lemma: formal character of simple module}
For each $ \mu \in \mathfrak{h}^*_q $, the formal character of the simple highest weight module $V(\mu)$ can be expressed as
\begin{equation}\label{eq: character of V}
\ch(V(\mu)) = \sum_{\lambda \in \mathfrak{h}^*_q} m_{\lambda,\mu}\ch(M(\lambda)) =\sum_{\lambda \in \mathfrak{h}^*_q} m_{\mu,\lambda}e^\lambda p
\end{equation}
where $m_{\mu,\lambda}$ are integers such that $m_{\mu,\mu}=1$ and $m_{\mu,\lambda}=0$ unless $\lambda\leq \mu$ and $\lambda$ is $\hat{W}$-linked to $\mu$.
\end{lemma}

\begin{rem}
If $ \mu \in  \weights^+ $  the set of dominant integral weights, then \cite[Proposition 4.4]{voigt2017complex} gives a more precise formula than \eqref{eq: character of V}.
\end{rem}

\begin{cor}\label{coro: formal character of general module}
For each $M\in \O$, there exists a finite set $\{\mu_1,\ldots, \mu_m\}\subset \mathfrak{h}^*_q$ such that
\begin{equation}\label{eq: character of general}
\ch(M)=\sum_{i=1}^m\sum_{\lambda \in \mathfrak{h}^*_q}[M:V(\mu_i)]m_{\mu_i,\lambda}e^\lambda p.
\end{equation}
where $m_{\mu_i,\lambda}$ are integers such that $m_{\mu_i,\mu_i}=1$ and $m_{\mu_i,\lambda}=0$ unless $\lambda\leq \mu_i$ and $\lambda$ is $\hat{W}$-linked to $\mu_i$.
\end{cor}

\begin{rem}
Since $\hat{W}$ is a finite group, the sums on the right hand side of \eqref{eq: character of V} and \eqref{eq: character of general} are both finite.
\end{rem}

\subsection{Reduced rational expressions of formal characters of modules in $\O$}
Notice that we can write the formal character $p = \prod_{\beta\in \positiveroots} \left(\sum_{m=0}^\infty e^{-m\beta}\right)$ as 
\begin{equation}
  p =  \frac{1}{\prod_{\beta\in \positiveroots}(1-e^{-\beta})}
\end{equation}
so by Corollary \ref{coro: formal character of general module} for each $M\in \O$, we can write its formal character as
\begin{equation}\label{eq: unreduced rational formal character}
\ch(M)=\frac{\sum_{i=1}^m\sum_{\lambda \in \mathfrak{h}^*_q}[M:V(\mu_i)]m_{\mu_i,\lambda}e^\lambda}{\prod_{\beta\in \positiveroots}(1-e^{-\beta})}
\end{equation}

We want to simplify $\ch(M)$ to obtain a reduced fraction, which needs some work because the ring $\Chi$ is not a UFD.

 Let $\mathcal{S}$ be the ring of $\mathbb{Z}$-coefficient polynomials generated by $e^{-\alpha_i}$, $i=1,\ldots,N$, where $\{\alpha_1,\ldots, \alpha_N\}$ is the set of simple roots. It is clear that $\prod_{\beta\in \positiveroots}(1-e^{-\beta})\in \mathcal{S}$    but 
 $$\sum_{i=1}^m\sum_{\lambda \in \mathfrak{h}^*_q}[M:V(\mu_i)]m_{\mu_i,\lambda}e^\lambda$$
is not necessarily contained in $ \mathcal{S}$.

 We have the following definition.

\begin{defn}\label{def: reduced rational form}
Let $\Chi$ be as in Definition \ref{def: convolution ring}. We say that $a\in \Chi$ can be written in  reduced rational form if  there exists a subset $T_a\subset \positiveroots$ and a finite collection  $\{\mu_1,\ldots, \mu_m\}\subset \mathfrak{h}^*_q$ such that
\begin{equation}\label{eq: reduced form}
a=\frac{\sum_{i=1}^me^{\mu_i}f_i}{\prod_{\beta\in T_a}(1-e^{-\beta})^{n_{\beta}}}
\end{equation}
where
\begin{enumerate}
\item $\mu_i-\mu_j$ is not in the root lattice  $ \roots$ for each $i\neq j$;
\item  $f_i$ is a  polynomial in $\mathcal{S}$ with nonzero constant term for each $i$;
\item $n_{\beta}$ is a positive integer for each $\beta\in T_a$;
\item The numerator and denominator of \eqref{eq: reduced form} are coprime. More precisely, for each $\beta\in T_a$, there exists an $f_i$ in the numerator such that $1-e^{-\beta}$ is not a factor of $f_i$.
\end{enumerate}
We call the set $T_a$ the denominator roots of $a$.
\end{defn}

\begin{lemma}\label{lemma: reduced rational form is unique}
For any $a\in \Chi$, the reduced rational form of $a$ is unique if exists.
\end{lemma}
\begin{proof}
It is clear from the definition and the fact that the polynomial ring $\mathcal{S}$ is a UFD.
\end{proof}

Not all elements in $\Chi$ can be written in reduced rational form. Nevertheless for formal characters of modules in $\O$ we have the following result.

\begin{lemma}\label{lemma: reduce rational form of formal character}
Let $M\in \O$ be a nonzero module. Then the set $T_{\ch(M)}$ of  denominator roots exists and $\ch(M)$ can be written uniquely in reduced rational form. Moreover we have
\begin{equation}\label{eq: reduced form of formal character}
\ch(M)=\frac{\sum_{i=1}^me^{\mu_i}f_i}{\prod_{\beta\in T_{\ch(M)}}(1-e^{-\beta})}
\end{equation}
and Property 1, 2, 3, 4 in Definition \ref{def: reduced rational form} are satisfied with all $n_{\beta}=1$. In the sequel we will denote  $T_{\ch(M)}$ by $T_M$ and we will call  $T_M$ the set of denominator roots of $M$.
\end{lemma}
\begin{proof}
It is a direct consequence of Corollary \ref{coro: formal character of general module}, Lemma \ref{lemma: reduced rational form is unique}, and \eqref{eq: unreduced rational formal character}.
\end{proof}

\begin{example}
By Lemma \ref{lemma: reduce rational form of formal character}, for any $\alpha\in \trueroots^+$ the formal power series
$$
\frac{1}{(1-e^{-\alpha})^2}
$$
cannot be the formal character of any module in $\O$ although $\frac{1}{(1-e^{-\alpha})^2}\in \Chi$. Intuitively this is because the multiplicity of $e^{-n\alpha}$ grows too fast as $n$ grows and here we have a precise criterion of this fact.
\end{example}

 \begin{cor}\label{coro: finite dimension and reduced rational form}
A nonzero module $M\in \O$ is finite dimensional if and only if its reduced rational form has denominator $=1$, i.e. $T_M=\emptyset$.
\end{cor}
\begin{proof}
It is implied by the uniqueness of reduced rational form.
\end{proof}

\begin{example}
For a Verma module $M(\mu)$, the reduced rational form of its formal character is 
$$
\ch(M(\mu))=\frac{e^{\mu}}{\prod_{\beta\in \positiveroots}(1-e^{-\beta})}
$$
hence $T_{M(\mu)}=\positiveroots$. 
\end{example}

\begin{example}
For a simple highest weight module $V(\mu)$ we have the reduced rational form
$$
\ch(V(\mu))=\frac{e^{\mu}f_{V(\mu)}}{\prod_{\beta\in T_{V(\mu)}}(1-e^{-\beta})}
$$
for some $f_{V(\mu)}\in \mathcal{S}$. In particular if $\mu\in \weights_q^+$ then $T_{V(\mu)}=\emptyset$ and $\ch(V(\mu))$ is given explicitly by the Weyl character formula. Actually \cite[Proposition 4.4]{voigt2017complex} only covers the  $\mu\in \weights^+$ case but the general case can be easily obtained by \cite[Lemma 2.41]{voigt2017complex}.
\end{example}

\begin{rem}
For any $M\in \O$ it is clear that the denominator roots $T_{M}\subset \bigcup T_{V(\mu)}$ where the union is for all $V(\mu)$ such that $[M:V(\mu)]\neq 0$. The author does not know whether we have $T_{M}= \bigcup T_{V(\mu)}$  for all $V(\mu)$ such that $[M:V(\mu)]\neq 0$. Nevertheless we do not need this result in this paper.
\end{rem}

\section{Tensor closed objects in category $\O$}
The category  $U_q(\mathfrak{g})$-Mod has a tensor product since $U_q(\mathfrak{g})$ is a Hopf algebra. Moreover we have the following lemma.
\begin{lemma}
 $U_q(\mathfrak{g})$-Mod is a braided category. In particular for any left $U_q(\mathfrak{g})$-modules $V$ and $W$ we have a $U_q(\mathfrak{g})$-module isomorphism $V\otimes W\cong W\otimes V$.
\end{lemma} 
\begin{proof}
It is clear since  $U_q(\mathfrak{g})$ is quasitriangular in the sense of \cite[Theorem 2.108]{voigt2017complex}. 
\end{proof}

However category $\O$ is not closed under tensor product.

\begin{defn}\label{def: tensor-closed objects}
We call a module $M\in \O$ \emph{tensor-closed} if for any $N\in\O$, the tensor product $M\otimes N\cong N\otimes M$ is still in $\O$.
\end{defn}

The following result is well-known.

\begin{lemma}\label{lemma: finite dimensional module is tensor-closed}
Any finite dimensional module $V\in \O$ is tensor-closed.
\end{lemma}
\begin{proof}
The proof is the same as that of \cite[Theorem 1.1 (d)]{humphreys2008representations}.
\end{proof}

In this section we prove the following result.

\begin{theorem}\label{thm: tensor-closed is finite dimensional}
A module $V\in \O$ is tensor-closed if and only if it is finite dimensional.
\end{theorem}

To give the proof more rigorously we introduce the following auxiliary category.

\begin{defn} \label{def: categoryO tilde}
A left module $ M $ over $ U_q(\mathfrak{g}) $ is said to belong to the category $ \widetilde{\O }$ if 
\begin{enumerate}[a)]
\item $ M $ is a weight module and all  weight spaces of $M$  are finite dimensional.
 
\item There exists finitely many weights $\nu_1,\ldots, \nu_l\in \mathfrak{h}^*_q$ such that supp$M\subset \bigcup_{i=1}^l(\nu_i-\roots^+)$, where supp$M=\{\lambda\in \mathfrak{h}^*_q\mid M_{\lambda}\neq 0\}$.
\end{enumerate}
Morphisms in category $\widetilde{\O} $ are all $ U_q(\mathfrak{g}) $-linear maps. 
\end{defn} 

It is clear that $\O$ is a full subcategory of $\widetilde{\O }$. $\widetilde{\O }$ is closed under tensor product and modules in $\widetilde{\O }$ have formal characters in the ring $\Chi$ in Definition \ref{def: convolution ring}. Moreover for $M$, $N\in \widetilde{\O }$ we have
\begin{equation}\ch(M\otimes N)=\ch(M)\ch(N).
\end{equation}

\begin{lemma}\label{lemma: V_mu times V_mu}
For two simple highest weight modules $V(\mu)$ and $V(\lambda)$, if $V(\mu)\otimes V(\lambda) \in \O$, then $$T_{V(\mu)}\cap T_{V(\lambda)}=\emptyset.$$ In particular for any infinite dimensional simple highest weight module $V(\mu)$ we have $V(\mu)\otimes V(\mu)\notin \O$.
\end{lemma}
\begin{proof}
For each simple highest weight module $V(\mu)$ we have the reduced rational form
$$
\ch(V(\mu))=\frac{e^{\mu}f_{V(\mu)}}{\prod_{\beta\in T_{V(\mu)}}(1-e^{-\beta})}
$$
where $f_{V(\mu)}$ is in the polynomial ring $\mathcal{S}$ such that $1-e^{-\beta}$ is not a factor of $f_{V(\mu)}$ for each $\beta\in T_{V(\mu)}$. Therefore
$$
\ch(V(\mu)\otimes V(\lambda))=\ch(V(\mu))\ch( V(\lambda))=\frac{e^{\mu+\lambda}f_{V(\mu)}f_{V(\lambda)}}{\prod_{\beta\in T_{V(\mu)}}(1-e^{-\beta})\prod_{\gamma\in T_{V(\lambda)}}(1-e^{-\gamma})}.
$$
Assume $ T_{V(\mu)}\cap T_{V(\lambda)}\neq \emptyset$ and let $\beta\in T_{V(\mu)}\cap T_{V(\lambda)}$, then $(1-e^{-\beta})^2$ appears in the denominator and $1-e^{-\beta}$ is not a factor of $f_{V(\mu)}$ nor $f_{V(\lambda)}$.  Therefore $(1-e^{-\beta})^2$ appears in the denominator of the reduced rational form of $\ch(V(\mu)\otimes V(\lambda))$. On the other hand by Lemma \ref{lemma: reduce rational form of formal character}, if $V(\mu)\otimes V(\lambda)$ is in $\O$ then the reduced rational form of $\ch(V(\mu)\otimes V(\lambda))$ cannot have squares in the denominator.  Hence  $V(\mu)\otimes V(\lambda)\in \O$  implies $T_{V(\mu)}\cap T_{V(\lambda)}=\emptyset$.

For  infinite dimensional $V(\mu)$, we know $T_{V(\mu)}\neq \emptyset$ by Corollary \ref{coro: finite dimension and reduced rational form}, so $V(\mu)\otimes V(\mu)\notin \O$.
\end{proof}

\begin{lemma}\label{lemma: V_mu times M_lambda}
For any infinite dimensional simple highest weight module $V(\mu)$  and any Verma module $M(\lambda)$ we have $V(\mu)\otimes M(\lambda)\cong M(\lambda)\otimes V(\mu)\notin \O$.
\end{lemma}
\begin{proof}
Similar to the proof of Lemma \ref{lemma: V_mu times V_mu}, we can show that $\ch(V(\mu)\otimes M(\lambda))$ has squares in the denominator of its reduced rational form.
\end{proof}

\begin{rem}
In general the product of two reduced rational forms needs not to be a reduced rational form since $\Chi$ is not a UFD. For example for any $\alpha \in \trueroots^+$ let 
$$
a=\frac{1+e^{-\alpha/2}}{1-e^{-\alpha}}, ~ b=\frac{1-e^{-\alpha/2}}{1-e^{-\alpha}}.
$$
It is clear that both $a$ and $b$ are reduced rational forms but
$$
ab=\frac{1-e^{-\alpha}}{(1-e^{-\alpha})^2}
$$
is not reduced. The author does not know if we restrict to formal characters of modules in $\O$, whether or not the product of reduced rational forms must be a reduced rational form. Nevertheless we do not need this result in this paper.
\end{rem}

\begin{proof}[Proof of Theorem \ref{thm: tensor-closed is finite dimensional}]
Let $M\in \O$ be infinite dimensional and we want to show that $M$ is not tensor-closed. Actually by Proposition \ref{prop: Jordan-Holder} there exists an infinite dimensional $V(\mu)$ in the Jordan-H\"{o}lder series of $M$. By Lemma \ref{lemma: V_mu times V_mu}, $V(\mu)\otimes V(\mu)\notin \O$. Since $\O$ is closed under subquotients, $M\otimes V(\mu)\notin \O$ too. So $M$ is not tensor-closed.
\end{proof}

\begin{rem}\label{rmk: counter example}
There exist two infinite dimensional  modules with tensor product still in $\O$.  Victor Ostrik gave the following example:
Let $\mathfrak{g}=sl(2)\oplus sl(2)$. Let $V$  be a Verma module for $U_q(sl(2))$ (with arbitrary highest weight). Using two projections $\mathfrak{g}\to sl(2)$ we can consider $V$ as a $U_q(\mathfrak{g})$-module in two different ways. Let us call the resulting $U_q(\mathfrak{g})$-modules $V_1$ and $V_2$. Then both $V_1$ and $V_2$ are in the category $\O$ for  $U_q(\mathfrak{g})$, and $V_1\otimes V_2$ is a Verma module of $U_q(\mathfrak{g})$.

It is an interesting question if $\mathfrak{g}$ is simple and $M$, $N\in \O$ are both infinite dimensional, is it always true that $M\otimes N\notin \O$. See \cite{wei2019tensor2} for a discussion in the cases that $\mathfrak{g}$ is simple of type ADE. 
\end{rem}

 \begin{rem}
All arguments and proofs in this paper work for the unquantized case as well. In particular we can proof Theorem \ref{thm: tensor-closed is finite dimensional} for the unquantized case using the method in this paper.
\end{rem}

\bibliography{tensorclosed}{}
\bibliographystyle{plain}
\end{document}